\documentclass{article}

\def\cal{\mathcal}
\usepackage{color}
\usepackage{amsthm,amssymb,amsfonts,mathrsfs}
\usepackage{amsmath}
\newtheorem{Theorem}{Theorem}[section]
\newtheorem{Proposition}{Proposition}[section]
\newtheorem{Lemma}{Lemma}[section]

\theoremstyle{definition}

\newtheorem{Definition}{Definition}[section]
\newtheorem{Remark}{Remark}

\newtheorem{Assumptions}{Hypothesis}[section]
\def\R{{\mathbb{R}}}

\def\vp{\varphi}
\def\ve{\varepsilon}

\def\fd{\mathfrak{d}}
\def\fg{\mathfrak{g}}
\def\fh{\mathfrak{h}}

\def\ds{\displaystyle}
\title {Carleman estimates for singular parabolic equations with interior degeneracy and non smooth coefficients}
\author{{\sc Genni Fragnelli}\thanks{Member of the Gruppo Nazionale per l'Analisi Matematica, la Probabilit\`a e le loro Applicazioni (GNAMPA) of the Istituto Nazionale di Alta Matematica (INdAM), and supported by the GNAMPA
project  2014 {\it Systems with irregular operators}.}\\
Dipartimento di Matematica\\ Universit\`{a} di Bari "Aldo Moro"\\
Via
E. Orabona 4\\ 70125 Bari - Italy\\ email: genni.fragnelli@uniba.it\\
{\sc Dimitri Mugnai }\thanks{Member of the Gruppo Nazionale per l'Analisi Matematica, la Probabilit\`a e le loro Applicazioni (GNAMPA) of the Istituto Nazionale di Alta Matematica (INdAM), and supported by the GNAMPA
project  2014 {\it Systems with irregular operators}. Research also supported by the M.I.U.R. project {\it Variational and perturbative aspects of nonlinear differential problems}.}\\
Dipartimento di Matematica e Informatica\\Universit\`a di
Perugia\\Via Vanvitelli 1, 06123 Perugia - Italy\\ email:
dimitri.mugnai@unipg.it}

\date{}
\begin{document}

\maketitle

\vspace{0.3cm}

\centerline{ {\it  }}

\begin{abstract}
We establish Carleman estimates for singular/degenerate parabolic Dirichlet problems with degeneracy and singularity occurring in the interior of the spatial domain. Our results are completely new, since this situation is not covered by previous contributions for degeneracy and singularity on the boundary. In addition, we consider non smooth coefficients, thus preventing the use of standard calculations in this framework.
\end{abstract}
%%%%%%%%%%%%%%%%%%%%%%%%%%%%%%%%%%%%
%%%%%%%%%%%%%%%%%%%%%%%%%%%%%%%%%%%
%%%%%%%%%%%%%%%%%%%%%%%%%%%%%%%%%%%%
%%%%%%%%%%%%%%%%%%%%%%%%%%%%%%%%%%%%
%%%%%%%%%%%%%%%%%%%%%%%%%%%%%%%%%%%%
%%%%%%%%%%%%%%%%%%%%%%%%%%%%%%%%%%%%%%%%
%%%%%%%%%%%%%%%%%%%%%%%%%%%%%%%%%%%%%%%


\noindent 2000AMS Subject Classification: 35Q93, 93B05, 93B07, 34H15, 35A23, 35B99

\section{Introduction}

Controllability issues for parabolic problems have been a mainstream topic in recent years, and several developments have been pursued: starting from the heat equation in bounded and unbounded domain,  related contributions have been found for more general situations. A common strategy in showing controllability results is to prove that certain global Carleman estimates hold true for the operator which is the adjoint of the given one.

In this paper we focus on a class of singular parabolic operators with interior degeneracy of
the form
\begin{equation}\label{2}
u_t - \left(a(x)u_x \right) _x - \displaystyle
\frac{\lambda}{b(x)}u,
\end{equation}
associated to Dirichlet boundary conditions and with $(t,x) \in
Q_T:=(0,T) \times (0,1)$, $T>0$ being a fixed number. Here $a$ and
$b$ degenerate at the same interior point $x_0 \in
(0,1)$, and $\lambda \in \R$ satisfies suitable assumptions
(see condition \eqref{lambda} below). The fact that both $a$ and $b$ degenerate at $x_0$ is just for the sake of simplicity and shortness: all the stated results are still valid if they degenerate at different points. The prototypes we have in mind are $a(x)=|x-x_0|^{K_1}$ and $b(x)=|x-x_0|^{K_2}$ for some $K_1,K_2>0$. The main goal is to establish global Carleman estimates
for operators of the form given in \eqref{2}.

Such estimates for uniformly parabolic operators without
degeneracies or singularities have been largely developed (see, e.g., Fursikov--Imanuvilov \cite{fi}).
Recently, these estimates have been also studied for operators which are
not uniformly parabolic. Indeed, as pointed out by several authors,
many problems coming from Physics (see \cite{KaZo}), Biology (see \cite{epma}) and Mathematical Finance
(see \cite{HW}) are described by degenerate parabolic equations. In particular, new Carleman estimates (and consequently null controllability properties) were established in \cite{acf}, and also in \cite{cmv0}, \cite{mv}, for the operator
\[
u_t - (au_x)_x + c(t,x)u, \quad (t,x) \in Q_T,
\]
where $a(0)=a(1)=0$, $a\in C^1(0,1)$ and $c \in L^\infty(Q_T)$ (see also \cite{cfr},
\cite{cfr1} or \cite{f} for problems in non divergence form).

An interesting situation
is the case of parabolic operators with singular inverse-square potentials.
First results in this direction were obtained in \cite{vz2} for the {\em non degenerate}
singular potentials with heat-like operator
 \begin{equation}\label{4}
  u_t -
\Delta u-\lambda\frac{1}{|x|^2}u, \quad (t,x) \in (0,T)\times \Omega, 
\end{equation}
with associated Dirichlet boundary conditions in a bounded domain $\Omega\subset \R^N$ containing the singularity $x=0$ in the interior (see also \cite{vz} for the wave and Schr\"{o}dinger equations and \cite{cazacu} for boundary singularity).
Similar operators of the form
\[
  u_t -
\Delta u-\lambda\frac{1}{|x|^{K_2}}u, \quad (t,x) \in (0,T)\times \Omega, 
\]
arise for example in quantum mechanics (see, for
example, \cite{bg}, \cite{d}), or in combustion problems (see, for
example, \cite{be}, \cite{bv}, \cite{dglv}, \cite{gv}), and is known to generate interesting phenomena. For example,
in \cite{bg} and in \cite{bg1} it was proved that, for all values of $\lambda$, global positive
solutions exist if ${K_2} < 2$, whereas
instantaneous and complete blow-up occurs if ${K_2} > 2$. In the critical case, i.e. ${K_2} = 2$,
the value of the parameter $\lambda$ determines the behavior of the
equation: if $\lambda \le {1}/{4}$ (which is the optimal constant of the
Hardy inequality, see \cite{b}) global positive solutions exist, while, if $\lambda >{1}/{4}$, instantaneous and complete blow-up occurs
(for other comments on this argument we refer to \cite{v1}). We recall that in \cite{vz2}, Carleman estimates were established for \eqref{4} under the
condition $\lambda \le \displaystyle\frac{1}{4}$. On the contrary,
if $\lambda >\displaystyle\frac{1}{4}$, in \cite{e} it was proved
that null controllability fails.

We remark that the non degenerate problems studied in \cite{bg,cazacu,e,v1,vz,vz2} cover the multidimensional case, while here we treat the case $N=1$, like Vancostenoble \cite{v1}, who studied the operator that
couples a degenerate diffusion coefficient with a singular
potential. In particular, for ${K_1}\in [0,2)$ and
${K_2}\leq 2-{K_1}$, the author established Carleman estimates for
the operator
\[
u_t - (x^{K_1} u_x)_x-\lambda\frac{1}{x^{K_2}}u, \quad (t,x) \in
Q_T,
\]
unifying the results of \cite{cmv} and \cite{vz2} in the purely
degenerate operator and in the purely singular one, respectively.
This result was then extended in \cite{fs} and in \cite{fs1} to the
operators
\begin{equation}\label{vecch}
u_t - (a(x) u_x)_x-\lambda\frac{1}{x^{K_2}}u, \quad (t,x) \in Q_T,
\end{equation}
for $a\sim x^{K_1},\, K_1\in[0,2)$  and ${K_2}\leq 2-{K_1}$. Here, as before, the
function $a$ degenerates at the boundary of the space domain, and Dirichlet boundary conditions are in force.

We remark the fact that all the papers cited so far, with the exception of \cite{e}, consider a singular/degenerate 
operator with degeneracy or singularity appearing at the boundary of the
domain. For example, in \eqref{vecch} as $a$ one can also consider the double power
function
\[a(x)= x^k(1-x)^\kappa, \quad x \,\in \,[0,1],\]
where $k$ and $\kappa$ are positive constants. To the best of our knowledge,  \cite{bfm1}, \cite{fm} and \cite{fm1} are the first
papers dealing with Carleman estimates (and, consequently, null
controllability) for operators (in divergence and in non divergence form with Dirichlet or Neumann boundary conditions) with mere degeneracy at the interior of the
space domain (for related systems of degenerate equations we refer to \cite{bfm}). We also recall \cite{gabriela} and \cite{gabriela1} for other type  of control problems associated to parabolic operators with interior degeneracy in divergence and non divergence form, respectively.

We emphasize the fact that an interior degeneracy does not imply a simple adaptation of  previous results and of the techniques used for boundary degeneracy. Indeed, imposing homogeneous Dirichlet boundary conditions, in the latter case one knows {\em a priori} that any function in the reference functional space vani\-shes exactly at the degeneracy point. Now, since the degeneracy point is in the interior of the spatial domain, such information is not valid anymore, and we cannot take advantage of this fact.

For this reason, the present paper is devoted to study the
operator defined in \eqref{2}, that couples a \textit{general}
degenerate diffusion coefficient with a \textit{general} singular
potential with degeneracy and singularity at the \textit{interior}
of the space domain. In particular, under suitable conditions on all
the parameters of the operator, we establish Carleman estimates and,
as a consequence, null controllability for the associated
generalized heat problem. Clearly, this result generalizes the
one obtained in \cite{fm} or \cite{fm1}: in fact, if $\lambda =0$ (that is, if
we consider the purely degenerate case), we recover the main contributions therein. See also \cite{JDE} for the problem in non divergence form for both Dirichlet and Neumann boundary conditions.

We also remark the fact that, though we have in mind prototypes as power functions for the degeneracy and the singularity, we don not limit our investigation to these functions, which are analytic out of their zero. Indeed, in this paper, pure powers singularities and degeneracies are considered only as a by--product of our main results, which are valid for {\em non smooth} general coefficients. This is quite a new view--point when dealing with Carleman estimates, since in this framework it is natural to assume that all the coefficients in force are quite regular. However, though this strategy has been successful for years, it is clear that also more irregular coefficients can be considered and appear in a natural way (for instance, see  \cite{garofalo}, \cite{kt}). Nevertheless, it will be clear from the proof that Carleman estimates {\em do} hold without particular conditions also in the non smooth setting, while for observability (and thus controllability) another technical condition is needed; however, such a condition is trivially true for the prototypes.

For this reason,  for the first time to our best knowledge, in \cite{fm1} {\em non smooth} degenerate coefficients were treated. Continuing in this direction, here we consider operators which contain both degenerate and singular coefficients, as in \cite{fs}, \cite{fs1} and \cite{v1}, but with low regularity. 

The classical approach to study singular operators in dimension 1 relies in the validity of the Hardy--Poincar\'e inequality
\begin{equation}\label{hp1}
\int_0^1 \frac{u^2}{x^2}dx\leq 4\int_0^1(u')^2dx,
\end{equation}
which is valid for every $u\in H^1(0,1)$ with $u(0)=0$. Similar inequalities are the starting point to prove well--posedness of the associated problems in the Sobolev spaces under consideration. In our situation, we prove an inequality related to \eqref{hp1}, but with a degeneracy coefficient in the gradient term; such an estimate is valid in a suitable Hilbert space ${\cal H}$ we shall introduce below, and it states the existence of $C>0$ such that for all $u\in {\cal H}$ we have
\[
\int_0^1\frac{u^2}{b}dx\leq C\int_0^1 a(u')^2dx.
\]
This inequality, which is related to another weighted Hardy-Poincar\'{e} inequality (see Proposition \ref{HP}), is the key step for the well--posedness of \eqref{linear}. Once this is done, global Carleman estimates follow, provided that an {\em ad hoc} choice of the weight functions is made (see Theorem \ref{Cor1}).

The introduction of the space ${\cal H}$ (which may coincide with the usual Sobolev space in some cases) is another feature of this paper, which is completely new with respect to all the previous approaches: including the integrability of $u^2/b$ in the definition of ${\cal H}$ has the advantage of obtaining immediately some useful functional properties, that in general could be hard to show in the usual Sobolev spaces. Indeed, solutions were already found in suitable function spaces for the ``critical'' and ``supercritical'' cases (when $\lambda$ equals or exceeds the best constant in the classical Hardy--Poincar\'e inequality) in \cite{vz}and \cite{VaZu} for purely singular problems. However, as already done in the purely degenerate case (\cite{acf,bfm,bfm1,cfr,cfr1,cmv,fs,fs1,f,fm,fm1,fggr}), a weighted Sobolev space must be used. For this reason, we believe that it is natural to unify these approaches in the singular/degenerate, as we do.

Now, let us consider the evolution problem
\begin{equation}\label{linear}
\begin{cases}
u_t - \left(au_x \right) _x  - \displaystyle \frac{\lambda}{b(x)}u=h(t,x) \chi_{\omega}(x), & (t,x) \in Q_T,\\
u(t,0)=u(t,1)=0, &t \in (0,T),\\
u(0,x)=u_0(x), & x \in (0,1),
\end{cases}
\end{equation}
where $u_0\in L^2(0,1)$,
the control $h \in L^2(Q_T)$ acts on a non empty interval $\omega\subset (0,1)$ and $\chi_\omega$ denotes the characteristic function of $\omega$. 

As usual, we  say that problem \eqref{linear} is null controllable if there exists $h\in L^2(Q_T)$ such that $u(T,x)\equiv 0$ for $x\in [0,1]$. A common strategy to show that \eqref{linear} is null controllable is to prove Carleman estimates for any solution $v$ of the adjoint problem of  \eqref{linear}
\[
\begin{cases}
v_t +(av_x)_x +\displaystyle \frac{ \lambda}{b(x)}v= 0, &(t,x) \in
Q_T,
\\[5pt]
v(t,0)=v(t,1) =0, & t \in (0,T),
\\[5pt]
v(T,x)= v_T(x),
\end{cases}
\]
and then deduce an observability inequality of the form
\begin{equation}\label{disob}
\int_0^1v^2(0,x) dx \le C_T\int_0^T \int_{\omega}v^2(t,x)dxdt,
\end{equation}
where $C_T>0$ is a universal constant. In the non degenerate case this has been obtained by a well--established procedure using Carleman and Caccioppoli inequalities. In our singular/degenerate non smooth situation, we need a new suitable Caccioppoli inequality (see Proposition \ref{caccio}), as well as global Carleman estimates in the non smooth non degenerate and non singular case (see Proposition \ref{classical Carleman}), which will be used far away from $x_0$ within a localization procedure via cut--off functions. Once these tools are established, we are able to prove an observability inequality like \eqref{disob}, and then controllability results for \eqref{linear}. However, we cannot do that in all cases, since we have to exclude that both the degeneracy and the singularity are strong, see condition {\bf(SSD)} below.

Finally, we remark that our studies with non smooth coefficients are particularly useful. In fact, though null controllability results could be obtained also in other ways, for example by a localization technique (at least when $x_0\in \omega$), in \cite{fm1} it is shown that with non smooth coefficients, even when $\lambda=0$, this is not always the case. For this, our approach with observability inequalities is very general and permits to cover more involved situations.

\medskip

The paper is organized in the following way: in Section \ref{sec2}
we study the well--posedness of problem \eqref{linear}, giving some general tools
that we shall use several times. In Section \ref{sec3} we provide one of
the main results of this paper, i.e. Carleman estimates for the adjoint problem to \eqref{linear}. In Section
\ref{sec4} we apply the previous Carleman estimates to prove an
observability inequality, which, together with a Caccioppoli type
inequality, lets us derive new null controllability results for the
associated singular/degenerate problem, also when the degeneracy and the
singularity points are {\em inside} the control region.

A final comment on the notation: by $c$ or $C$ we shall denote
{\em universal} positive constants, which are allowed to vary from line to
line.

\section{Well--posedness}\label{sec2}

The ways in which $a$ and $b$ degenerate at $x_0$ can be quite
different, and for this reason we distinguish four different types
of degeneracy. In particular, we consider the following cases:
\begin{Assumptions}\label{Ass0}
{\bf Doubly weakly degenerate case (WWD):} there exists $x_0\in
(0,1)$ such that $a(x_0)=b(x_0) =0$, $a, b>0$ on $[0, 1]\setminus
\{x_0\}$, $a, b\in W^{1,1}(0,1)$ and there exists $K_1, K_2 \in
(0,1)$ such that $(x-x_0)a' \le K_1 a$ and $(x-x_0)b' \le K_2 b$
a.e. in $[0,1]$.
\end{Assumptions}

\begin{Assumptions}\label{Ass0_1}
{\bf Weakly-strongly degenerate case (WSD):} there exists $x_0 \in
(0,1)$ such that $a(x_0)=b(x_0)=0$, $a, b>0$ on $[0, 1]\setminus
\{x_0\}$, $a\in W^{1,1}(0,1)$, $ b\in W^{1, \infty}(0,1)$ and there
exist $K_1\in (0,1)$, $K_2 \geq1$ such that $(x-x_0)a' \le K_1
a$ and $(x-x_0)b' \le K_2 b$ a.e. in $[0,1]$.
\end{Assumptions}

\begin{Assumptions}\label{Ass01_1}
{\bf Strongly-weakly degenerate case (SWD):} there exists $x_0 \in
(0,1)$ such that $a(x_0)=b(x_0)=0$, $a, b>0$ on $[0, 1]\setminus
\{x_0\}$, $a\in W^{1, \infty}(0,1)$, $b\in W^{1,1}(0,1)$, and there
exist $K_1\geq1$, $K_2 \in (0,1)$ such that $(x-x_0)a' \le K_1
a$ and $(x-x_0)b' \le K_2 b$ a.e. in $[0,1]$.
\end{Assumptions}

\begin{Assumptions}\label{Ass01}
{\bf Doubly strongly degenerate case (SSD):} there exists $x_0 \in
(0,1)$ such that $a(x_0)=b(x_0)=0$, $a, b>0$ on $[0, 1]\setminus
\{x_0\}$, $a, b\in W^{1, \infty}(0,1)$ and there exist $K_1, K_2 \geq1$ such that $(x-x_0)a' \le K_1 a$ and $(x-x_0)b' \le K_2 b$
a.e. in $[0,1]$.
\end{Assumptions}
Typical examples for the previous degeneracies and singularities are
$a(x)=|x- x_0|^{K_1}$ and $b(x)= |x-x_0|^{K_2}$, with $ 0<K_1,
K_2<2$.

\begin{Remark}
The restriction $K_i<2$ is related to the controllability issue. Indeed, it is clear from the proof of Theorem \ref{theorem1} that such a condition is useless, for example, when $\lambda<0$. On the other hand, concerning controllability, we will not consider the case $K_i\geq2$, since if $a(x)=|x- x_0|^{K_1}$, $K_1 \ge2$ and $\lambda=0$, by a standard change of variables  (see \cite{fm1}), problem \eqref{linear} may be transformed in a non degenerate heat equation
on an unbounded domain, while the control remains distributed in a bounded domain. This situation is now well--understood, and the lack of null controllability was proved by Micu and Zuazua in \cite{mz}.
\end{Remark}

We will use the following result several times; we state it for $a$, but an analogous one holds for  $b$ replacing $K_1$ with $K_2$:
\begin{Lemma}[Lemma 2.1, \cite{fm}]\label{Lemma 2.1}
Assume that there exists $x_0 \in (0,1)$ such that $a(x_0)=0$, $a>0$ on $[0, 1]\setminus
\{x_0\}$, and either
\begin{itemize}
\item $a\in W^{1, 1}(0,1)$ and there
exist $K_1\in (0,1)$ such that $(x-x_0)a' \le K_1
a$ a.e. in $[0,1]$, or
\item $a\in W^{1, \infty}(0,1)$ and there
exist $K_1\in [1,2)$ such that $(x-x_0)a' \le K_1
a$ a.e. in $[0,1]$.
\end{itemize}
\begin{enumerate}
\item Then for all $\gamma \ge K_1$ the map
$$
\begin{aligned}
& x \mapsto \dfrac{|x-x_0|^\gamma}{a} \mbox { is non increasing on
the left of } x=x_0 \\
& \mbox{and non decreasing on the right of }
x=x_0,\\
&\mbox{ so that }\lim_{x\to x_0}\dfrac{|x-x_0|^\gamma}{a}=0 \mbox{
for all }\gamma>K_1.
\end{aligned}
$$
\item If $K_1<1$, then
    $\displaystyle\frac{1}{a} \in L^{1}(0,1)$.\\
\item If $K _1\in[1,2)$, then $\displaystyle \frac{1}{\sqrt{a}} \in
    L^{1}(0,1)$ and $\displaystyle \frac{1}{a}\not \in L^1(0,1)$.
\end{enumerate}
\end{Lemma}

For the well--posedness of the problem, we start introducing the following
weighted Hilbert spaces, which are suitable to study all situations, namely the {\em (WWD)}, {\em (SSD)}, {\em
(WSD)} and {\em (SWD)} cases:
\[
\begin{aligned}
 H^1_{a} (0,1):=\Big\{ u \in W^{1,1}_0(0,1) \,:\, \sqrt{a} u' \in  L^2(0,1)\Big\}
\end{aligned}
\]
and
\[
H^1_{a,b}(0,1):=\left\{ u \in H^1_{a}(0,1)\,:\,  \dfrac{u}{\sqrt{b}}\in L^2(0,1)\right\},
\]
endowed with the inner products
\[
\langle u,v\rangle_{H^1_a(0,1)}:=\int_0^1au'v'dx+\int_0^1uv\,dx,\] and
\[
\langle u,v\rangle_{H^1_{a,b}(0,1)}=\int_0^1au'v'dx+\int_0^1uv\,dx+\int_0^1\frac{uv}{b}dx,
\]
respectively.

Note that, if $u\in H^1_a(0,1)$, then $au'\in L^2(0,1)$, since $|au'|\leq(\ds\max_{[0,1]}\sqrt{a})\sqrt{a}|u'|$.

We recall the following weighted Hardy--Poincar\'e inequality, see \cite[Proposition 2.6]{fm}:
\begin{Proposition}\label{HP}
Assume that $p \in C([0,1])$, $p>0$ on $[0,1]\setminus \{x_0\}$,
$p(x_0)=0$ and there exists $q>1$ such that the function
\begin{equation}\label{condp}
\begin{aligned}
x \mapsto \dfrac{p(x)}{|x-x_0|^{q}} &\mbox { is
non increasing on the left of } x=x_0 \\
& \mbox{ and non
decreasing on the right of } x=x_0.
\end{aligned}
\end{equation}
\noindent Then, there exists a constant $C_{HP}>0$ such that for any
function $w$, locally absolutely continuous on $[0,x_0)\cup (x_0,1]$
and satisfying
$$
w(0)=w(1)=0 \,\, \mbox{with } \int_0^1 p(x)|w^{\prime}(x)|^2 \,dx <
+\infty,
$$ the following inequality holds:
\begin{equation}\label{hardy1}
\int_0^1 \dfrac{p(x)}{(x-x_0)^2}w^2(x)\, dx \leq C_{HP}\, \int_0^1
p(x) |w^{\prime}(x)|^2 \,dx.
\end{equation}
\end{Proposition}

\begin{Remark}
Actually, such a proposition was proved in \cite{fm} also requiring $q<2$. However, as it is clear from the proof, the result is true without such an upper bound on $q$, that in \cite{fm} was used for other estimates.
\end{Remark}

Moreover, we will also need other types of Hardy's inequalities. Let us start with the following crucial
\begin{Lemma}\label{L2'}
If $K_1+K_2\leq 2$ and $K_2< 1$, then there exists a constant $C>0$ such that
\begin{equation}\label{1}
\int_0^1\frac{u^2}{b}dx\leq
 C \int_0^1 a (u')^2dx
\end{equation}
for every $u\in H^1_a(0,1)$.
\end{Lemma}
\begin{proof}
We set $p(x):=\dfrac{(x-x_0)^2}{b}$, so that $p$ satisfies \eqref{condp} with $q=2-K_2>1$ by Lemma \ref{Lemma 2.1}. Thus, taken $u\in H^1_a(0,1)$, by Proposition \ref{HP}, we get
\[
\int_0^1\frac{u^2}{b}dx=\int_0^1\frac{p(x)}{(x-x_0)^2}u^2dx\leq C_{HP}\int_0^1p(x)|u'(x)|^2dx.
\]
Now, by Lemma \ref{Lemma 2.1},
\[
p(x)=(x-x_0)^{2-K_1-K_2}a(x)\frac{(x-x_0)^{K_1}}{a(x)}\frac{(x-x_0)^{K_2}}{b(x)}\leq ca(x)
\]
for some $c>0$, and the claim follows.
\end{proof}
\begin{Remark}
A similar proof shows that, when $K_1+2K_2\leq 2$ and $K_2<1/2$, then
\[
\int_0^1\frac{u^2}{b^2}dx\leq
 C \int_0^1 a (u')^2dx
\]
for every $u\in H^1_a(0,1)$.
\end{Remark}

Lemma \ref{L2'} implies that $H^1_a(0,1)= H^1_{a,b}(0,1)$ when $K_1+K_2\leq 2$ and $K_2<1$. However, inequality \eqref{1} holds in other cases, see Proposition \ref{PropH} below. In order to prove such a proposition, we need a preliminary result:
\begin{Lemma}\label{leso}
If $K_2\geq1$, then $u(x_0)=0$ for every $u\in H^1_{a,b}(0,1)$.
\end{Lemma}
\begin{proof}
Since $u\in W^{1,1}_0(0,1)$, there exists $\lim_{x\to x_0}u(x)=L\in \R$. If $L\neq 0$, then $|u(x)|\geq \ds\frac{L}{2}$ in a neighborhood of $x_0$, that is
\[
\frac{|u(x)|^2}{b}\geq \frac{L^2}{4b}\not \in L^1(0,1)
\]
by Lemma \ref{Lemma 2.1}, and thus $L=0$.
\end{proof}

We also need the following result, whose proof, with the aid of Lemma \ref{leso}, is a simple adaptation of the one given in \cite[Lemma 3.2]{fggr}.
\begin{Lemma}\label{densita}
If $K_2\geq1$, then
\[
H^1_c(0,1):=\Big\{u\in H^1_0(0,1)\mbox{ \rm such that supp}\,u\subset (0,1)\setminus\{x_0\}\Big\}
\] 
is dense in $H^1_{a,b}(0,1)$.
\end{Lemma}

In the spirit of \cite[Lemma 5.3.1]{davies}, now we are ready for the following ``classical'' Hardy inequality in the space $H^1_{a,b}(0,1)$ for $a(x)=|x-x_0|^\alpha$ and $b(x)= |x-x_0|^{2-\alpha}$. However, note that our inequality is more interesting than the classical one, since we admit a singularity inside the interval:
\begin{Lemma}\label{hardissimo}
For every $\alpha\in \R$ the inequality
\[
\frac{(1-\alpha)^2}{4}\int_0^1 \frac{u^2}{|x-x_0|^{2-\alpha}}dx\leq  \int_0^1|x-x_0|^\alpha (u')^2dx
\]
holds true for every $u\in H^1_{|x-x_0|^\alpha,|x-x_0|^{2-\alpha}}(0,1)$.
\end{Lemma}
\begin{proof} The case $\alpha=1$ is trivial. So, take $\beta=(1-\alpha)/2\neq0$ and $\ve\in(0,1-x_0)$. \\
{\sl First case: $\beta<0$ ($\alpha>1)$}. In this case we have
\[
\begin{aligned}
&\int_{x_0+\ve}^1(x-x_0)^\alpha(u')^2dx\\
&=\int_{x_0+\ve}^1(x-x_0)^\alpha\Big((x-x_0)^\beta\big((x-x_0)^{-\beta}u\big)'+\beta(x-x_0)^{-1}u\Big)^2dx\\
&\geq\beta^2  \int_{x_0+\ve}^1(x-x_0)^{\alpha-2}u^2dx+2\beta \int_{x_0+\ve}^1(x-x_0)^{\alpha+\beta-1}u\big((x-x_0)^{-\beta}u\big)'dx\\
&=\beta^2  \int_{x_0+\ve}^1(x-x_0)^{\alpha-2}u^2dx+\beta\big((x-x_0)^{-\beta}u)^2\Big|^1_{x_0+\ve} \mbox{ (since $\alpha+\beta-1=-\beta$)}\\
& \geq \beta^2  \int_{x_0+\ve}^1(x-x_0)^{\alpha-2}u^2dx.\\
\end{aligned}
\]

Letting $\ve\to 0^+$, we get that
\begin{equation}\label{adx}
\int_{x_0}^1(x-x_0)^\alpha(u')^2dx\geq \beta^2 \int_{x_0}^1(x-x_0)^{\alpha-2}u^2dx.
\end{equation}

{\sl Second case: $\beta>0$}. In this situation we have $2-\alpha>1$.  Thus, in view of Lemma \ref{densita} with $K_2=2-\alpha$, we will prove \eqref{adx} first if $u\in H^1_c(0,1)$ and then, by density, if $u\in H^1_{|x-x_0|^\alpha,|x-x_0|^{2-\alpha}}(0,1)$. Thus, take $u\in H^1_c(0,1)$; proceeding as above, we get
\[
\begin{aligned}
&\int_{x_0+\ve}^1(x-x_0)^\alpha(u')^2dx\\
&\geq\beta^2  \int_{x_0+\ve}^1(x-x_0)^{\alpha-2}u^2dx+\beta\big((x-x_0)^{-\alpha}u)^2\Big|^1_{x_0+\ve}\\
& \geq \beta^2  \int_{x_0+\ve}^1(x-x_0)^{\alpha-2}u^2dx,
\end{aligned}
\]
since $u(x_0+\ve)=0$ for $\ve$ small enough.

Passing to the limit as $\ve\to0^+$, and using Lemma \ref{densita}, we get that \eqref{adx} holds true for every $u\in H^1_{|x-x_0|^\alpha,|x-x_0|^{2-\alpha}}(0,1)$.

Operating in a symmetric way on the left of $x_0$, we get the conclusion.
\end{proof}

As a corollary of the previous result, we get the following improvement of Lemma \ref{L2'}.
\begin{Proposition}\label{PropH}
If one among Hypotheses $\ref{Ass0}, \ref{Ass0_1}, \ref{Ass01_1}$ holds with $K_1+K_2\leq 2$, then \eqref{1} holds for every $u\in H^1_{a,b}(0,1)$.
\end{Proposition}
\begin{proof}
By Lemma \ref{Lemma 2.1} and Lemma \ref{hardissimo} with $\alpha= 2- K_2$, we immediately get that for every $u\in H^1_{a,b}(0,1)$,
\[
\begin{aligned}
\int_0^1\frac{u^2}{b}dx&\leq c\int_0^1\frac{u^2}{|x-x_0|^{K_2}}dx\leq c\int_0^1|x-x_0|^{2-K_2}(u')^2dx \\
&\leq c\int_0^1|x-x_0|^{K_1}(u')^2dx\leq c\int_0^1a(u')^2dx.
\end{aligned}
\]
\end{proof}

\begin{Remark}
It is well known that when $K_1=K_2=1$, an inequality of the form \eqref{1} doesn't hold (see \cite{mu}). Being such an inequality fundamental for the observability inequality (see Lemma \ref{obser.regular}), it is no surprise if with our techniques we cannot handle this case in Section \ref{sec4}.
\end{Remark}

The fundamental space in which we will work is clearly the one where the Hardy--Poincar\'e--type inequality \eqref{1} holds: in view of Proposition, it is clear that such a space is 
\[
{\cal H}:=
H^1_{a,b}(0,1)
\]

\begin{Remark}\label{rem1}

Under the assumptions of Proposition \ref{PropH},  the standard norm $\|\cdot\|_{\cal H}^2$ is equivalent to
\[
\|u\|_{\circ}^2:= \int_0^1 a (u')^2 dx
\]
for all $u \in {\cal H}$. Indeed, for all $u \in {\cal H}
$, we have
\[
\int_0^1 u^2 dx = \int_0^1 b \frac{u^2}{b}dx \le c \int_0^1 a (u')^2
dx,
\]
and this is enough to conclude.

Moreover, when $\lambda <0$, an equivalent norm is given by
\[
\|u\|_{\sim}^2:= \int_0^1 a (u')^2 dx-\lambda \int_0^1 \frac{u^2}{b}dx.
\]
This is particularly useful if Hypothesis \ref{Ass01} holds (see the proof of Theorem \ref{theorem1}).
\end{Remark}

First, let us call $C^*$ the best constant of \eqref{1} in ${\cal H}$. From now on, we make the following assumptions on $a$, $b$ and $\lambda$:
\begin{Assumptions}\label{Ass03}
\begin{enumerate}
\item One among Hypothesis $\ref{Ass0}$,  $\ref{Ass0_1}$  or $\ref{Ass01_1}$ holds true with $K_1+K_2\leq 2$, and
we assume that
\begin{equation}\label{lambda}
\lambda\in \left(0,  \frac{1}{C^*}\right),
\end{equation}
or
\item Hypotheses $\ref{Ass0}$,  $\ref{Ass0_1}$, $\ref{Ass01_1}$ or $\ref{Ass01}$ hold with $\lambda<0$.
\end{enumerate}
\end{Assumptions}
Observe that the assumption $\lambda \neq 0$ is not restrictive, since the case $\lambda =0$ was already considered in \cite{fm} and in \cite{fm1}.

Using the previous lemmas one can prove the next inequality.
\begin{Proposition}\label{eq}
Assume Hypothesis $\ref{Ass03}$. Then there exists $\Lambda\in(0,1]$ such
that for all $u \in{\cal H}$
\[
\int_0^1 a(u')^2 dx - \lambda \int_0^1 \frac{u^2}{b} dx \ge \Lambda
\int_0^1 a(u')^2 dx.
\]
\end{Proposition}
\begin{proof} If $\lambda< 0$, the result is obvious taking $\Lambda=1$. Now, assume that $\lambda \in
\ds\left(0, \frac{1}{C^*}\right)$. Then
\[
\begin{aligned}
\int_0^1 a(u')^2 dx - \lambda \int_0^1 \frac{u^2}{b} dx
\ge \int_0^1 a(u')^2 dx - \lambda C^*\int_0^1 a(u')^2
dx \ge \Lambda\int_0^1 a(u')^2 dx.
\end{aligned}
\]
\end{proof}

We recall the following definition:
\begin{Definition}\label{def_div}
Let $u_0 \in L^2(0,1)$ and $h \in L^2(Q_T)$. A function $u$
is said to be a (weak) solution of \eqref{linear} if
\[
u \in L^2(0, T; {\cal H}) \cap H^1([0, T];{\cal H}^*)
\]
and it satisfies  \eqref{linear} in the sense of ${\cal H}^*$-valued distributions.
\end{Definition}
Note that, by \cite[Lemma 11.4]{rr}, any solution belongs to $C([0,T];L^2(0,1))$.

Finally, we introduce the Hilbert space
\[
H^2_{a, b}(0,1) :=  \Big\{ u \in H^1_a(0,1)\, :\, au' \in H^1(0,1) \mbox{ and }Au
\in L^2(0,1)\Big\},
\]
where
\[
Au:=\left(au' \right)'  + \displaystyle \frac{\lambda}{b}u \mbox{ with }D(A) =H^2_{a, b}(0,1).
\]

\begin{Remark}\label{aul2}
Observe that if $u \in D(A)$, then $\ds \frac{u}{b}$ and $\ds \frac{u}{\sqrt{b}}\in L^2(0,1)$, so that $u\in H^1_{a,b}(0,1)$ and inequality \eqref{1} holds.
\end{Remark}

We also recall the following integration by parts
with functions in the reference spaces:
\begin{Lemma}[Green formula, \cite{fggr}, Lemma 2.3]\label{green}
Assume one among the Hypo\-the\-ses $\ref{Ass0},\,\ref{Ass0_1},\,\ref{Ass01_1},\,\ref{Ass01}$. Then, for all $(u,v)\in
H^2_{a,b}(0,1)\times H^1_a(0,1)$ the following identity holds:
\begin{equation}\label{greenformula}
\int_0^1(au')' v dx= - \int_0^1 au'v' dx.
\end{equation}
\end{Lemma}

Observe that in the non degenerate case, it is well known that the heat operator with an inverse--square singular potential
\[
u_t-\Delta u-\lambda \frac{u}{|x|^2}v
\]
gives rise to well--posed Cauchy-Dirichlet problems if and only if $\lambda$ is not larger than the best Hardy inequality (see \cite{bg1}, \cite{CaMa}, \cite{VaZu}). For this reason, it is not strange that we require an analogous condition for problem \eqref{linear}, by invoking Hypothesis \ref{Ass03}; as a consequence, using the standard semigroup theory, we have that \eqref{linear} is well--posed:
\begin{Theorem}\label{theorem1}
Assume Hypothesis $\ref{Ass03}$. For every $u_0\in L^2(0,1)$ and $h\in L^2(Q_T)$ there exists  a unique solution of problem \eqref{linear}.
In particular, the ope\-ra\-tor $A: D(A) \to
L^2(0,1)$ is non positive and self-adjoint in $L^2(0,1)$ and it gene\-rates an analytic contraction
semigroup of angle $\pi/2$. Moreover, let $u_0 \in D(A)$; then
\[
\begin{aligned}
h\in W^{1,1}(0,T;L^2(0,1)) &\Rightarrow u \in C^1(0,T; L^2(0,1)) \cap C([0,T];D(A)),\\
h\in L^2(Q_T) &\Rightarrow u \in H^1(0,T;L^2(0,1)).
\end{aligned}
\]
\end{Theorem}
\begin{proof}Observe that $D(A)$ is dense in $L^2(0,1)$. The existence of the unique solution follows in a standard way by a Faedo--Galerkin procedure, see, e.g., \cite[Theorem 11.3]{rr}, or \cite[Theorem 3.4.1 and Remark 3.4.3]{LM}. Let us prove the other facts.

\textbf{$\boldsymbol {A}$ is non positive.} By Proposition \ref{eq}, Remark \ref{rem1} and Lemma \ref{green},
for all $u \in D(A)$ we have
\[
\begin{aligned}
- \langle Au, u \rangle_{L^2(0,1)} \!=\! -\int_0^1\!\left( (au')' +
\frac{\lambda}{b}u\right)u \,dx\! =\! \int_0^1 \!a (u')^2 dx - \lambda
\int_0^1\! \frac{u^2}{b}dx
 \ge C
\|u\|_{\cal H}^2.
\end{aligned}
\]

\textbf{$\boldsymbol{ A}$ is self-adjoint.} Let $T:L^2(0,1)\to L^2(0,1)$ be the mapping defined in the following usual way: to each $h\in L^2(0,1)$ associate the weak solution $u=T(h)\in {\cal H}$ of 
\[
\int_0^1 \left(au'v' - \lambda \frac{u v}{b}
\right) dx= \int_0^1 h v \,dx
\]
for every $v\in {\cal H}$. Note that $T$ is well defined by the Lax--Milgram Lemma via Proposition \ref{eq}, which also implies that $T$ is continuous. Now, it is easy to see that $T$ is injective and symmetric. Thus it is self--adjoint. As a consequence, $A=T^{-1}:D(A)\to L^2(0,1)$ is self--adjoint (for example, see \cite[Proposition A.8.2]{Taylor}).

\textbf{$\boldsymbol{A}$ is $\boldsymbol{m}$--dissipative}. Being $A$ non positive and self--adjoint, this is a straightforward consequence of \cite[Corollary 2.4.8]{ch}. Then $(A, D(A))$ generates a cosine family
and an analytic contractive semigroup of angle $\ds\frac{\pi}{2}$ on
$L^2(0,1)$ (see, for instance, \cite[Examples 3.14.16 and
3.7.5]{abhn}).

The additional regularity is a consequence of \cite[Lemma 4.1.5 and Proposition 4.1.6]{ch} in the first case, and of \cite[6.2.2 and 6.2.4]{arendt} in the second one.
\end{proof}

\section{Carleman estimates for singular/degenerate
problems}\label{sec3}

In this section we prove one of the main result of this paper, i.e.
a new Carleman estimate with boundary terms for solutions of the
singular/degenerate problem
\begin{equation}\label{P-adjoint}
\begin{cases}
v_t +\left(av_x \right) _x  + \displaystyle \frac{\lambda}{b(x)}v=h(t,x)=h, & (t,x) \in Q_T,\\
v(t,0)=v(t,1)=0, &t \in (0,T),\\
v(T,x) = v_T(x),
\end{cases}
\end{equation}
which is the adjoint of problem
\eqref{linear}. 

On the degenerate function $a$ we make the following assumption:
\begin{Assumptions}\label{Ass02}
Hypothesis $\ref{Ass03}$ holds. Moreover, if $K_1 > \displaystyle \frac{4}{3}$, then 
there exists a constant $\theta \in \left(0, K_1\right]$ such that
\begin{equation}\label{dainfinito_1}
\begin{array}{ll}
x \mapsto \dfrac{a(x)}{|x-x_0|^{\theta}} &
\begin{cases}
& \mbox{ is non increasing on the left of $x=x_0$,}\\
& \mbox{ is non decreasing on the right of $x=x_0$}.
\end{cases}
\end{array}
\end{equation}
In addition, when $ K_1 >\displaystyle \frac{3}{2}$ the function in  \eqref{dainfinito_1}  is bounded below away from $ 0$
and there exists a constant $\Sigma>0$ such that
\begin{equation}\label{Sigma}
|a'(x)|\leq \Sigma |x-x_0|^{2\theta-3} \mbox{ for a.e. }x\in
[0,1].
\end{equation}
Moreover, if $\lambda <0$ we require that 
\begin{equation}\label{ipob}
(x-x_0)b'(x) \ge 0  \text{ in } [0,1].
\end{equation}
\end{Assumptions}

\begin{Remark}
If $a(x)= |x-x_0|^{K_1}$, then \eqref{dainfinito_1} is clearly satisfied with $\theta=K_1$. Moreover, the additional requirements for the sub-case $ K_1 >\displaystyle \frac{3}{2}$
are technical
ones and are introduced in \cite{fm1} to guarantee the convergence of some
integrals (see \cite[Appendix]{fm1}). Of course,
the prototype $a(x)=|x-x_0|^{K_1}$ satisfies again such conditions with
$\theta=K_1$. Finally, \eqref{ipob} is clearly satisfied by the prototype $b(x)= |x-x_0|^{K_2}$.
\end{Remark}

To prove Carleman estimate, let us introduce the function $\varphi:= \Theta \psi$, where
\begin{equation}\label{c_1}
\Theta(t):=\frac{1}{[t(T-t)]^4} \quad \text{and} \quad \psi(x) :=
c_1\left[\int_{x_0}^x \frac{y-x_0}{a(y)}dy- c_2\right],
\end{equation}
where $c_2> \displaystyle \ds\sup_{[0,1]}\int_{x_0}^x \frac{y-x_0}{a(y)}dy$ and $c_1>0$ (for the observability inequality $c_1$ will be taken sufficiently large, see Lemma \ref{lemma3}). Observe that $\Theta (t) \rightarrow + \infty \,
\text{ as } t \rightarrow 0^+, T^-$, and clearly $-c_1c_2 \le \psi < 0$.

The main result of this section is the following
\begin{Theorem}\label{Cor1}
Assume Hypothesis $\ref{Ass02}$. Then, there exist two positive
constants $C$ and $s_0$, such that every solution $v$ of
\eqref{P-adjoint} in
\begin{equation}\label{v}
\mathcal{V}:=L^2\big(0, T; H^2_{a,b}(0,1)\big) \cap H^1\big(0,
T;{\cal H}\big)
\end{equation}
satisfies, for all $s \ge s_0$,
\[
\begin{aligned}
&\int_{Q_T} \left(s\Theta a(v_x)^2 + s^3 \Theta^3
\frac{(x-x_0)^2}{a} v^2\right)e^{2s\varphi}dxdt\\
&\le C\left(\int_{Q_T} h^2e^{2s\varphi}dxdt +
sc_1\int_0^T\left[a\Theta e^{2s \varphi(t,x)}(x-x_0)(v_x)^2
dt\right]_{x=0}^{x=1}\right).
\end{aligned}
\]
\end{Theorem}

\begin{Remark}
In \cite{vz2} the authors prove a related Carleman inequality for the {\em non degenerate} singular 1-D problem
\begin{equation}\label{vancozu}
\begin{cases}
\ds v_t+v_{xx}+\frac{\mu}{x^2}+\frac{\lambda}{x^\beta}v=h & (t,x)\in Q_T,\\
v(t,0)=v(t,1)=0& t\in (0,T),\\
v(T,x)=v_T(x)& x\in (0,1),
\end{cases}
\end{equation}
where $\beta\in[0,2)$. When $\mu=0$ and $x_0=0$, such an inequality reads as follows:
\[
\int_{Q_T} \left(s^3\Theta^3 x^2v^2 +\frac{s}{2}\Theta \frac{v^2}{x^2}+\frac{s}{2} \Theta \frac{v^2}{x^{2/3}}\right)e^{2s\Psi}dxdt\leq \frac{1}{2} \int_{Q_T} h^2e^{2s\Psi}dxdt,
\]
where $\Psi(x)= \ds\frac{x^2}{2}-1<0$ in $[0,1]$. Actually, it is proved for solutions $v$ such that
\begin{equation}\label{suppcomp}
\mbox{$v(t, x) = 0 $ for all $(t, x)\in (0, T ) \times (1- \eta, 1)$ for some $\eta\in (0,1)$}.
\end{equation}
However, in \cite[Remark 3.5]{vz2} the authors say that Carleman estimates can be proved also for all solutions of \eqref{vancozu} not satisfying \eqref{suppcomp}. We think that this latter situation is much more interesting, since by the Carleman estimates, if $h=0$, then $v\equiv0$ even if \eqref{suppcomp} does not hold.
\end{Remark}

The proof of Theorem \ref{Cor1} is quite long, and several
intermediate lemmas will be used. First, for $s> 0$, define the function
\[
w(t,x) := e^{s \varphi (t,x)}v(t,x),
\]
where $v$ is any solution of \eqref{P-adjoint} in $\mathcal{V}$;
observe that, since $v\in\mathcal{V}$ and $\vp<0$, then
$w\in\mathcal{V}$ and satisfies
\begin{equation}\label{1'}
\begin{cases}
(e^{-s\varphi}w)_t + \left(a(e^{-s\varphi}w)_x \right) _x+ \lambda \displaystyle\frac{e^{-s\varphi}w} {b} =h, & (t,x) \in (0,T) \times (0,1),\\
w(t,0)=w(t,1)=0, &  t \in (0,T),\\ w(T,x)= w(0, x)= 0, & x \in
(0,1).
\end{cases}
\end{equation}
As usual, we re--write the previous problem as
follows: setting
\[
Lv:= v_t + (av_x)_x + \lambda \frac{v}{b} \quad \text{and} \quad
%\]
%and
%\[
L_sw= e^{s\varphi}L(e^{-s\varphi}w),
\]
then \eqref{1'} becomes
\[
\begin{cases}
L_sw= e^{s\varphi}h,\\
w(t,0)=w(t,1)=0, & t \in (0,T),\\
w(T,x)= w(0, x)= 0, & x \in (0,1).
\end{cases}
\]
Computing $L_sw$, one has
\[
\begin{aligned}
L_sw
%& = (\partial_t - s\varphi_t)w + (\partial_x - s
%\varphi_x)[a(x)(\partial_x - s \varphi_x)w] \\
%&= (a(x)w_x)_x
% - s \varphi_t w + s^2a(x) \varphi_x^2 w + w_t -2sa(x)\varphi_x w_x -
% s(a(x)\varphi_x)_xw\\
% &
=L^+_sw + L^-_sw,
\end{aligned}
\]
where
\[
L^+_sw := (aw_x)_x + \lambda\frac{w}{b}
 - s \varphi_t w + s^2a \varphi_x^2 w,
\]
and
\[
L^-_sw := w_t -2sa\varphi_x w_x -
 s(a\varphi_x)_xw.
\]

Of course,
\begin{equation}\label{stimetta}
\begin{aligned}
2\langle L^+_sw, L^-_sw\rangle &\le 2\langle L^+_sw, L^-_sw\rangle+
\|L^+_sw \|_{L^2(Q_T)}^2 + \|L^-_sw\|_{L^2(Q_T)}^2\\
& =\| L_sw\|_{L^2(Q_T)}^2= \|he^{s\varphi}\|_{L^2(Q_T)}^2,
\end{aligned}
\end{equation}
where $\langle\cdot, \cdot \rangle$ denotes the scalar product in
$L^2(Q_T)$. As usual, we will separate the scalar product $\langle
L^+_sw, L^-_sw\rangle$ in distributed terms and boundary terms.

\begin{Lemma}\label{lemma1}
The following identity holds:
\begin{equation}\label{D&BT}
\begin{aligned}
& \left.
\begin{aligned}
&\langle L^+_sw,L^-_sw\rangle \;\\&=\; \frac{s}{2} \int_{Q_T}
\varphi_{tt} w^2dxdt 
- 2s^2 \int_{Q_T}a \varphi_x \varphi_{tx}w^2dxdt \\& +s
\int_{Q_T}(2 a^2\varphi_{xx} + aa'\varphi_x)(w_x)^2 dxdt
\\&+ s^3 \int_{Q_T}(2a \varphi_{xx} + a'\varphi_x)a
(\varphi_x)^2 w^2dxdt
\\&-s\lambda \int_{Q_T}
\frac{a\varphi_xb'}{b^2} w^2 dxdt
\end{aligned}\right\}\;\text{\{D.T.\}} \\
& \left.
\begin{aligned}
& + \int_0^T[aw_xw_t]_{x=0}^{x=1} dt- \frac{s}{2}
\int_0^1[w^2\varphi_t]_{t=0}^{t=T}dx+ \frac{s^2}{2}\int_0^1
[a(\varphi_x)^2 w^2]_{t=0}^{t=T}dt\\& + \int_0^T[-s\varphi_x
(aw_x)^2 +s^2a\varphi_t \varphi_x w^2 - s^3 a^2(\varphi_x)^3w^2-
s\lambda \frac{a\varphi_x }{b} w^2]_{x=0}^{x=1}dt
\\&+ \int_0^T[-sa(a\varphi_x)_xw
w_x]_{x=0}^{x=1}dt-\frac{1}{2} \int_0^1
\Big[a(w_x)^2-\lambda\frac{1}{2b}w^2 \Big]_{t=0}^{t=T}dx.
\end{aligned}\right\}\; \text{\{B.T.\}}
\end{aligned}
\end{equation}
\end{Lemma}
\begin{proof} Computing $\langle L^+_sw,L^-_sw\rangle$,
one has that
\[
\langle L^+_sw,L^-_sw\rangle = I_1+ I_2 + I_3+I_4,
\]
where
\[
\begin{aligned}
I_1& := \int_{Q_T}\big((aw_x)_x
 - s \varphi_t w + s^2a (\varphi_x)^2 w\big) w_t  dxdt,\\
I_2& := \int_{Q_T}\big((aw_x)_x
 - s \varphi_t w + s^2a (\varphi_x)^2 w\big) (-2sa\varphi_xw_x)  dxdt,\\
I_3& := \int_{Q_T}\big((aw_x)_x
 - s \varphi_t w + s^2a (\varphi_x)^2 w\big)(-s(a\varphi_x)_xw) dxdt,
\end{aligned}
\]
and
\[
I_4:= \lambda \int_{Q_T} \frac{w}{b}\big( w_t -2sa\varphi_x
w_x -
 s(a\varphi_x)_xw\big)dxdt.
\]
 By several integrations by parts in space and in time (see
\cite[Lemma 3.4]{acf}, \cite[Lemma 3.1]{fm} or \cite[Lemma 3.1]{fm1}), and observing that $  \int_{Q_T}a (a \varphi_x)_{xx} w w_xdxdt=0$ (by the very definition of $\varphi$), we get
\begin{equation}\label{D&BT_prima}
\begin{aligned}
&I_1+I_2+I_3 \;\\&=\; \frac{s}{2} \int_{Q_T} \varphi_{tt}
w^2dxdt 
- 2s^2 \int_{Q_T}a \varphi_x \varphi_{tx}w^2dxdt\\&  +s
\int_{Q_T}(2 a^2\varphi_{xx} + aa'\varphi_x)(w_x)^2 dxdt
\\&+ s^3 \int_{Q_T}(2a \varphi_{xx} + a'\varphi_x)a
(\varphi_x)^2 w^2dxdt
\\
& + \int_0^T[aw_xw_t]_{x=0}^{x=1} dt- \frac{s}{2}
\int_0^1[w^2\varphi_t]_{t=0}^{t=T}dx+ \frac{s^2}{2}\int_0^1
[a(\varphi_x)^2 w^2]_{t=0}^{t=T}dt\\& + \int_0^T[-s\varphi_x (aw_x)^2
+s^2a\varphi_t \varphi_x w^2 - s^3 a^2(\varphi_x)^3w^2
]_{x=0}^{x=1}dt
\\&+ \int_0^T[-sa(a\varphi_x)_xw
w_x]_{x=0}^{x=1}dt-\frac{1}{2} \int_0^1 \Big[a(w_x)^2\Big]_{t=0}^{t=T}dx.
\end{aligned}
\end{equation}
Next, we compute $I_4$:
\begin{equation}\label{T_4}
\begin{aligned}
I_4& = \lambda \left(\int_{Q_T} \frac{1}{2b}(w^2)_t dxdt - 2
s \int_{Q_T} \frac{a}{b} \varphi_xw_xwdxdt \right.\\
& \left.-s \int_{Q_T}
\frac{(a\varphi_x)_x}{b} w^2 dxdt\right)\\
&= \lambda \left( \int_0^1 \frac{1}{2b}[w^2]_{t=0}^{t=T} dx -  s
\int_{Q_T} \frac{a}{b} \varphi_x (w^2)_x dxdt -s
\int_{Q_T} \frac{(a\varphi_x)_x}{b} w^2 dxdt\right)\\
&= \lambda \left( \int_0^1 \frac{1}{2b}[w^2]_{t=0}^{t=T} dx -  s
\int_0^T\left[\frac{a}{b} \varphi_x w^2\right]_{x=0}^{x=1}dt\right.\\
& \left. + s \int_{Q_T} \left(\frac{a \varphi_x}{b}\right)_x
w^2 dxdt-s
\int_{Q_T} \frac{(a\varphi_x)_x}{b} w^2 dxdt\right)\\
&=\lambda \left( \int_0^1 \frac{1}{2b}[w^2]_{t=0}^{t=T} dx -  s
\int_0^T\left[\frac{a \varphi_x}{b} w^2\right]_{x=0}^{x=1}dt -s
\int_{Q_T} \frac{a\varphi_xb'}{b^2} w^2 dxdt\right).
\end{aligned}
\end{equation}
Adding \eqref{D&BT_prima}-\eqref{T_4}, \eqref{D&BT} follows
immediately.
\end{proof}

For the boundary terms in \eqref{D&BT}, we have:
\begin{Lemma}\label{lemma4}
The boundary terms in \eqref{D&BT} reduce to
\[-s \int_0^{T} \left[ \Theta
     (aw_{x})^2 \psi' \right]_{x=0}^{x=1}dt.
\]
\end{Lemma}
\begin{proof}
As in \cite{fm} or \cite{fm1}, using the definition of $\varphi$ and the boundary
conditions on $w$, one has that
\begin{equation}\label{bd1}
\begin{aligned}
& \int_0^T[aw_xw_t]_{x=0}^{x=1} dt- \frac{s}{2}
\int_0^1[w^2\varphi_t]_{t=0}^{t=T}dx+ \frac{s^2}{2}\int_0^1
[a(\varphi_x)^2 w^2]_{t=0}^{t=T}dt\\& + \int_0^T[-s\varphi_x (aw_x)^2
+s^2a\varphi_t \varphi_x w^2 - s^3 a^2(\varphi_x)^3w^2
]_{x=0}^{x=1}dt
\\&+ \int_0^T[-sa(a\varphi_x)_xw
w_x]_{x=0}^{x=1}dt-\frac{1}{2} \int_0^1 \Big[a(w_x)^2\Big]_{t=0}^{t=T}dx= -s
\int_0^{T} \left[ \Theta
     (aw_{x})^2 \psi' \right]_{x=0}^{x=1}dt.
\end{aligned}
\end{equation}
Moreover, since $w \in \mathcal{V}$, $ w\in C\big([0, T];{\cal H}
\big)$; thus $w(0, x)$, $w(T,x)$ are well defined, and using
the boundary conditions of $w$, we get that
\[
\int_0^1\left[\frac{1}{2b}w^2\right]_{t=0}^{t=T} dx=0.
\]

Now, consider the last boundary term $\displaystyle s \lambda
\int_0^T \left[\frac{a\varphi_x}{b}w^2\right]_{x=0}^{x=1}dt$. Using
the definition of $\varphi$, this term becomes $\displaystyle s
\lambda \int_0^T
\left[\Theta\frac{a\psi'}{b}w^2\right]_{x=0}^{x=1}dt$. By
definition of $\psi$, the
function $\displaystyle \Theta \frac{a\psi'}{b}w^2$ is bounded  in $(0,T)$.
Thus, by the boundary conditions on $w$, one has
\[
s \lambda \int_0^T
\left[\Theta\frac{a\psi'}{b}w^2\right]_{x=0}^{x=1}dt=0.
\]
\end{proof}

Now, the crucial step is to prove the following estimate:
\begin{Lemma}\label{lemma2}
Assume Hypothesis $\ref{Ass02}$. Then there exist two positive constants
$s_0$ and $C$ such that for
all $s \ge s_{0}$ the distributed terms of \eqref{D&BT} satisfy the
estimate
\[
\begin{aligned}
&\frac{s}{2}\int_{Q_T} \varphi_{tt} w^2 dxdt- 2s^2 \int_{Q_T}a
\varphi_x \varphi_{tx}w^2 dxdt\\& + s \int_{Q_T}(2
a^2\varphi_{xx} + aa' \varphi_x)(w_x)^2 dxdt \\&+ s^3 \int_0^T
\int_0^1(2a \varphi_{xx}+ a' \varphi_x)a (\varphi_x)^2 w^2
dxdt-s\lambda \int_{Q_T} \frac{a\varphi_xb'}{b^2} w^2
dxdt\\&\ge \frac{C}{2}s\int_{Q_T} \Theta a(w_x)^2 dxdt +
\frac{C^3}{2}s^3 \int_{Q_T}\Theta^3 \frac{(x-x_0)^2}{a} w^2
dxdt.
\end{aligned}
\]
\end{Lemma}
\begin{proof}
Proceeding as in \cite[Lemma 3.2]{fm} or in \cite[Lemma 4.1]{fm1}, one can prove
that, for $s $ large enough,
\[
\begin{aligned}
&\frac{s}{2} \int_{Q_T} \varphi_{tt} w^2dxdt - 2s^2 \int_{Q_T}a \varphi_x \varphi_{tx}w^2dxdt  \\&+s
\int_{Q_T}(2 a^2\varphi_{xx} + aa'\varphi_x)(w_x)^2 dxdt
\\&+ s^3 \int_{Q_T}(2a \varphi_{xx} + a'\varphi_x)a
(\varphi_x)^2 w^2dxdt \\
&\ge \frac{3C}{4}s\int_{Q_T} \Theta a(w_x)^2 dxdt +
\frac{C^3}{2}s^3 \int_{Q_T}\Theta^3 \frac{(x-x_0)^2}{a} w^2
dxdt,
\end{aligned}
\]
where $C$ is a positive constant. Let us remark that one can assume $C$ as large as desired, provided that $s_0$ increases as well. Indeed, taken $k>0$, from
\[
Cs{\cal A}_1+C^3s^3{\cal A}_2=kC\frac{s}{k}{\cal A}_1+k^3C^3\frac{s^3}{k^3}{\cal A}_2,
\]
we can choose $s_0'=ks_0$ and $C'=kC$ large as needed.

Now, we estimate the term $\displaystyle -s\lambda \int_{Q_T}
\frac{a\varphi_xb'}{b^2} w^2 dxdt$. If $\lambda <0$, the thesis follows immediately by the previous inequality and by \eqref{ipob}. Otherwise, if $\lambda >0$, by definition of
$\varphi$ and the assumption on $b$, one has
\[
\begin{aligned}
-s\lambda \int_{Q_T} \frac{a\varphi_xb'}{b^2} w^2 dxdt &=
-s\lambda \int_{Q_T} \Theta \frac{a\psi'b'}{b^2} w^2
dxdt\\&=-s \lambda c_1 \int_{Q_T} \Theta
\frac{(x-x_0)b'}{b^2}w^2 dxdt\\
&\ge -s\lambda c_1K_2 \int_{Q_T} \frac{\Theta}{b}w^2 dxdt.
\end{aligned}
\]
Since $w(t,\cdot)\in {\cal H}$ for every $t\in[0,1]$, for $w\in {\cal V}$, by \eqref{1} we get
\[
\int_{Q_T} \frac{\Theta}{b}w^2 dxdt \le
C^* \int_{Q_T} \Theta a(w_x)^2 dxdt.
\]
Hence,
\[
-s\lambda \int_{Q_T} \frac{a\varphi_xb'}{b^2} w^2 dxdt \ge
-s\lambda c_1K_2 C^* \int_{Q_T} \Theta
a(w_x)^2 dxdt,
\]
and we can assume, in view of what remarked above, that this last
quantity is greater than
\[
 - s\frac{C}{4}\int_{Q_T} \Theta a(w_x)^2 dxdt.
\]
Summing up, the distributed terms of $\int_{Q_T}L^+_s w L^-_s
w dxdt$ can be estimated as
\[
\{\text{D.T.} \}\ge \frac{C}{2}s\int_{Q_T} \Theta a(w_x)^2
dxdt + \frac{C^3}{2}s^3 \int_{Q_T}\Theta^3 \frac{(x-x_0)^2}{a}
w^2 dxdt,
\]
for $s$ large enough and $C >0$.
\end{proof}

From Lemma \ref{lemma1}, Lemma \ref{lemma4} and Lemma \ref{lemma2},
we deduce immediately that there exist two positive constants $C$
and $s_0$, such that for
all $s \ge s_0$,
\begin{equation}\label{D&BT1}
\begin{aligned}
\int_{Q_T}L^+_s w L^-_s w dxdt &\ge Cs\int_{Q_T} \Theta
a(w_x)^2 dxdt\\&+ Cs^3 \int_{Q_T}\Theta^3 \frac{(x-x_0)^2}{a}
w^2 dxdt- s\int_0^{T} \left[\Theta
a^2w_{x}^{2}\psi'\right]_{x=0}^{x=1}dt.
\end{aligned}
\end{equation}

Thus, a straightforward consequence of \eqref{stimetta} and of
\eqref{D&BT1} is the next result.
\begin{Lemma}
Assume Hypothesis $\ref{Ass02}$. Then, there exist two
positive constants $C$ and $s_0$, such that for all $s \ge s_0$, \begin{equation}\label{2stelle}
\begin{aligned}
   & s\int_{Q_T} \Theta a(w_x)^2dxdt  + s^3
\int_{Q_T}\Theta^3 \frac{(x-x_0)^2}{a} w^2 dxdt\\&\le
    C\left(\int_{Q_T} h^2 e^{2s\varphi(t,x)}dxdt+ s\int_0^{T} \left[\Theta
       a^2(w_{x})^{2}\psi'\right]_{x=0}^{x=1}dt.
    \right).
    \end{aligned}
    \end{equation}
\end{Lemma}
Recalling the definition of $w$, we have $v= e^{-s\varphi}w$ and
$v_{x}= -s\Theta \psi'e^{-s\varphi}w + e^{-s\varphi}w_{x}$. Thus,
substituting in \eqref{2stelle}, Theorem \ref{Cor1} follows.

\section{Observability results and application to null
controllability}\label{sec4}
In this section we shall apply the just established Carleman inequalities to observability and controllability issues. For this, we assume that the control set $\omega$ satisfies
the following assumption:
\begin{Assumptions} \label{ipotesiomega}
The subset $\omega$ is such that
\begin{enumerate}
\item[(i)] it is an interval which contains the degeneracy point:
\begin{equation}\label{omega1}
\omega=(\alpha,\beta) \subset (0,1) \mbox{ is such that $x_0 \in
\omega$},
\end{equation}
or
\item[(ii)] it is an interval lying on one
side of the degeneracy point:
\begin{equation}\label{omega}
\omega=(\alpha,\beta) \subset (0,1) \mbox{ is such that $x_0\not \in
\bar \omega$}.
\end{equation}
\end{enumerate}
\end{Assumptions}

On the coefficients $a$ and $b$ we essentially start with the assumptions made so far, with the exception of Hypothesis \ref{Ass01}, and we add another technical one. We summarize all of them in the following:
\begin{Assumptions}\label{Ass04}$\ $
\begin{itemize}
\item Assume one among Hypotheses \ref{Ass0}, \ref{Ass0_1} or \ref{Ass01_1} with $K_1+K_2\leq2$ and $\lambda<1/C^*$. 
\item If $\lambda<0$, \eqref{ipob} holds.
\item If $K_1>4/3$, condition \eqref{dainfinito_1} holds, and if $K_1>3/2$, \eqref{Sigma} is satisfied.
\item If Hypothesis $\ref{Ass0}$ or $\ref{Ass0_1}$ holds, there exist
two functions $\fg \in L^\infty_{\rm loc}([0,1]\setminus \{x_0\})$, $\fh \in W^{1,\infty}_{\rm loc}([0,1]\setminus \{x_0\})$ and
two strictly positive constants $\fg_0$, $\fh_0$ such that $\fg(x) \ge \fg_0$ for a.e. $x$ in $[0,1]$ and
\begin{equation}\label{5.3'}
-\frac{a'(x)}{2\sqrt{a(x)}}\left(\int_x^B\fg(t) dt + \fh_0 \right)+ \sqrt{a(x)}\fg(x) =\fh(x, B)
\end{equation}
for a.e.$x,B \in [0,1]$ with $x<B<x_0$ or $x_0<x<B$.
\end{itemize}
\end{Assumptions}

\begin{Remark}
Since we require identity \eqref{5.3'} far from $x_0$, once $a$ is given, it is easy to find $\fg,\fh,\fg_0$ and $\fh_0$ with the desired properties.  For example, if $a(x):= |x-x_0|^\alpha, \alpha \in (0,1)$,  
 we can take    $ \fg (x)\equiv \fg_0=\fh_0=1 $ and $\fh (x,B)= \ds |x-x_0|^{\frac{\alpha}{2}-1}\left[ \frac{\alpha}{2}\text{sign} (x-x_0) (B+1-x) + |x-x_0|\right]$, for all $x$ and $B \in [0,1]$, with $x<B<x_0$ or $x_0<x<B$.
 Clearly, $\fg \in L^\infty_{\rm loc}([0,1]\setminus \{x_0\})$ and $\fh \in W^{1,\infty}_{\rm loc}([0,1]\setminus \{x_0\}; L^\infty(0,1))$.
\end{Remark}

Now, we associate to problem \eqref{linear} the
homogeneous adjoint problem
\begin{equation}\label{h=0}
\begin{cases}
v_t +(av_x)_x +\displaystyle \frac{ \lambda}{b(x)}v= 0, &(t,x) \in
Q_T,
\\[5pt]
v(t,0)=v(t,1) =0, & t \in (0,T),
\\[5pt]
v(T,x)= v_T(x),
\end{cases}
\end{equation}
where $T>0$ is given and $v_T(x) \in L^2(0,1)$. By the Carleman estimate in Theorem
\ref{Cor1}, we will deduce the following observability inequality
for all the degenerate cases:
\begin{Proposition}
\label{obser.} Assume Hypotheses $\ref{ipotesiomega}$ and $\ref{Ass04}$. Then
there exists a positive constant $C_T$ such that every solution $v
\in  C([0, T]; L^2(0,1)) \cap L^2 (0,T; {\cal H})$ of
\eqref{h=0} satisfies
 \begin{equation}\label{obser1.}
\int_0^1v^2(0,x) dx \le C_T\int_0^T \int_\omega v^2(t,x)dxdt.
\end{equation}
\end{Proposition}

Using the observability inequality \eqref{obser1.} and a standard
technique (e.g., see \cite[Section 7.4]{LRL}), one can prove the
null controllability result for the linear degenerate problem
\eqref{linear}:
\begin{Theorem}\label{th3}
Assume Hypotheses $\ref{ipotesiomega}$ and $\ref{Ass04}$. Then, given $u_0 \in L^2
(0,1)$, there exists $h \in L^2(Q_T)$ such that the solution $u$ of
\eqref{linear} satisfies
\begin{equation*}
u(T,x)= 0 \ \text{ for every  } \  x \in [0, 1].
\end{equation*}
Moreover
\[
\int_{Q_T} h^2 dxdt \le C \int_0^1 u_0^2(x) dx,
\]
for some positive constant $C$.
\end{Theorem}

\subsection{Proof of Proposition \ref{obser.}} In
this subsection we will prove, as a consequence of the Carleman
estimate proved in Section \ref{sec3}, the observability inequality
\eqref{obser1.}. For this purpose, we will give some preliminary
results. As a first step, consider the adjoint problem
\begin{equation}\label{h=01}
\begin{cases}
v_t +Av= 0, &(t,x) \in Q_T,
\\[5pt]
v(t,0)=v(t,1) =0, & t \in (0,T),
\\[5pt]
v(T,x)= v_T(x) \,\in D({A}^2),
\end{cases}
\end{equation}
where
\[
D({A}^2) = \Big\{u \,\in \,D({ A })\,:\, A u \,\in \,D( A )
\;\Big\}
\]
and ${ A }u:=(au_x)_x+ \lambda \displaystyle\frac{u}{b}$. Observe
that $D({ A }^2)$ is densely defined in $D({ A })$ for the graph norm (see, for
example, \cite[Lemma 7.2]{b}) and hence in $L^2(0,1)$. As in
\cite{cfr}, \cite{cfr1}, \cite{f} or \cite{fm}, define the following
class of functions:
\[
\cal{W}:=\Big\{ v\text{ is a solution of \eqref{h=01}}\Big\}. \]
Obviously (see, for example, \cite[Theorem 7.5]{b})
\[ \cal{W}\subset
C^1\big([0,T]\:;\:H^2_{a,b}(0,1)\big) \subset \mathcal{V} \subset
\cal{U},
\]
where, $\mathcal{V}$ is defined in \eqref{v} and
\begin{equation}\label{U}
\cal{U}:= C([0,T]; L^2(0,1)) \cap L^2(0, T; {\cal H}).
\end{equation}

We start with
\begin{Proposition}[Caccioppoli's inequality]\label{caccio}Assume Hypothesis $\ref{Ass03}$.
Let $\omega'$ and $\omega$ be two open subintervals of $(0,1)$ such
that $\omega'\subset \subset \omega \subset  (0,1)$ and $x_0 \not
\in \bar\omega'$. Let $\varphi(t,x)=\Theta(t)\Upsilon(x)$, where
$\Theta$ is defined in \eqref{c_1} and
\[
\Upsilon \in C([0,1],(-\infty,0))\cap
C^1([0,1]\setminus\{x_0\},(-\infty,0))
\]
is such that
\begin{equation}\label{stimayx}
|\Upsilon_x|\leq \frac{c}{\sqrt{a}} \mbox{ in }[0,1]\setminus\{x_0\}
\end{equation}
for some $c>0$. Then, there exist two positive constants $C$ and
$s_0$ such that every solution $v \in \cal W$ of the adjoint problem
\eqref{h=01} satisfies
\begin{equation}\label{lemme-caccio}
   \int_{0}^T \int _{\omega'}   (v_x)^2e^{2s\varphi } dxdt
    \ \leq \ C \int_{0}^T \int _{\omega}   v^2  dxdt,
\end{equation}
for all $s\geq s_0$.
\end{Proposition}
Of course, our prototype for $\Upsilon$ is the function $\psi$
defined in \eqref{c_1}, since
\[
|\psi'(x)|= c_1 \sqrt{\frac{|x-x_0|^2}{a(x)}}\frac{1}{\sqrt{a(x)}}\leq c \frac{1}{\sqrt{a(x)}}
\]
by Lemma \ref{Lemma 2.1}.

\begin{proof} The proof follows the one of \cite[Proposition
4.2]{fm}, but it is different for the presence of the singular term.

Let us consider a smooth function $\xi: [0,1] \to \R$ such that
  \[\begin{cases}
    0 \leq \xi (x)  \leq 1, &  \text{for all } x \in [0,1], \\
    \xi (x) = 1 ,  &   x \in \omega', \\
    \xi (x)=0, &     x \in [\, 0, 1 ]\setminus \omega.
    \end{cases}\]
Since $v$ solves \eqref{h=01}, we have
\begin{equation}\label{ca}
    \begin{aligned}
    0 &= \int _0 ^T \frac{d}{dt} \left(\int _0 ^1 \xi ^2 e^{2s\varphi}
    v^2dx\right)dt
    =  \int_{Q_T}2s \xi ^2  \varphi _t e^{2s\varphi} v^2 + 2 \xi ^2
    e^{2s\varphi} vv_t dxdt
    \\
    &= 2 \int_{Q_T} \xi ^2 s\varphi _t e^{2s\varphi} v^2dxdt + 2 \int_{Q_T}\xi
    ^2e^{2s\varphi} v \left(-\lambda \frac{v}{b}-(a v_x)_x\right)dxdt
    \\
    &= 2 \int_{Q_T}\xi ^2 s\varphi _t e^{2s\varphi} v^2 dxdt -  2\lambda\int_{Q_T} \xi^2 e^{2s \varphi}\frac{v^2}{b} dx dt+2 \int_{Q_T}( \xi
    ^2e^{2s\varphi} v )_x a v_xdxdt.
\end{aligned}
    \end{equation}
If $\lambda <0$, then, differentiating the last term in \eqref{ca}, we get
\[
\begin{aligned}
2\int_{Q_T}\xi^2e^{2s\vp}a(v_x)^2dxdt&=2\lambda\int_{Q_T} \xi^2 e^{2s \varphi}\frac{v^2}{b} dx dt-2 \int_{Q_T}\xi ^2 s\varphi _t e^{2s\varphi} v^2 dxdt\\
&-2\int_{Q_T}(\xi^2e^{2s\vp})_xavv_xdxdt\\
&\leq -2 \int_{Q_T}\xi ^2 s\varphi _t e^{2s\varphi} v^2 dxdt-2\int_{Q_T}(\xi^2e^{2s\vp})_xavv_xdxdt,
\end{aligned}
\]
and then one can proceed as for the proof of \cite[Proposition 4.2]{fm}, obtaining the claim.

Otherwise, if $\lambda >0$, fixed $\ve>0$, by the Cauchy--Schwarz inequality, we have for $w=\xi e^{s\varphi}v$
    \[
    \begin{aligned}
  \int_0^1 \xi^2 e^{2s \varphi}\frac{v^2}{b} dx &\le  C^*\int_0^1 a (w_x)^2 dx\\
    & \leq C_\ve\int_0^1 a [(\xi e^{s\varphi})_x]^2 v^2 dx +\ve\int_0^1 \xi^2e^{2s\varphi} a(v_x)^2 dx
    \end{aligned}
    \]
for some $C_\ve>0$. Moreover,
\[
[(\xi e^{s\varphi})_x]^2 \le C\chi_\omega(e^{2s\varphi}+ s^2 (\varphi_x)^2e^{s\varphi}) \le C\chi_\omega\left(1+ \frac{1}{a}\right)
    \]
for some positive constant $C$. Indeed, $e^{2s\vp}<1$, while $s^2(\vp_x)^2e^{2s\vp}$ can be
estimated with
\[
\frac{c}{(-\max \Upsilon)^2}(\Upsilon_x)^2 \leq \frac{c}{a}
\]
by \eqref{stimayx}, for some constants $c>0$.
Thus
    \begin{equation}\label{ca1}
    \begin{aligned}
2\lambda  \int_{Q_T} \xi^2 e^{2s \varphi}\frac{v^2}{b} dxdt
    & \le  2\lambda C_{\ve}\int_{Q_T} a [(\xi e^{s\varphi})_x]^2 v^2 dxdt\\
    &+  2\lambda\ve\int_{Q_T} \xi^2e^{2s\varphi} a(v_x)^2 dxdt\\
    &\le  C\int_0^T\int_\omega v^2 dxdt+ 2\lambda\ve\int_{Q_T} \xi^2e^{2s\varphi} a(v_x)^2 dxdt,
    \end{aligned}
    \end{equation}
for a positive constant $C$ depending on $\ve$.
Hence,  differentiating the last term in \eqref{ca} and using \eqref{ca1}, we get
\[
\begin{aligned}
2\int_{Q_T}\xi^2e^{2s\vp}a(v_x)^2dxdt&=2\lambda\int_{Q_T} \xi^2 e^{2s \varphi}\frac{v^2}{b} dx dt-2 \int_{Q_T}\xi ^2 s\varphi _t e^{2s\varphi} v^2 dxdt\\
&-2\int_{Q_T}(\xi^2e^{2s\vp})_xavv_xdxdt\\
&\leq  C\int_0^T\int_\omega  v^2 dxdt+ 2\lambda\ve\int_{Q_T} \xi^2e^{2s\varphi} a(v_x)^2 dxdt\\
&-2 \int_{Q_T}\xi ^2 s\varphi _t e^{2s\varphi} v^2 dxdt-2\int_{Q_T}(\xi^2e^{2s\vp})_xavv_xdxdt.
\end{aligned}
\] 
Thus, applying again the Cauchy-Schwarz inequality, we get
\[
\begin{aligned}
&(2-  2\lambda\ve)\int_{Q_T}\xi^2e^{2s\vp}a(v_x)^2dxdt\le C\int_0^T\!\!\!\int_\omega v^2 dxdt-2 \int_{Q_T}\!\!\!\!\xi ^2 s\varphi _t e^{2s\varphi} v^2 dxdt\\
&-2\int_{Q_T}(\xi^2e^{2s\vp})_xavv_xdxdt\\
&
 \le C\int_0^T\int_\omega v^2 dxdt- 2 \int_0^T \int_\omega \xi ^2 s\varphi _t e^{2s\varphi} v^2dxdt
    +2\ve \int_0^T \int_\omega\left( \sqrt{a} \xi e^{s\varphi} v_x \right) ^2dxdt
 \\
 &+ D_\ve\int_0^T \int_\omega\left( \sqrt{a} \frac{( \xi ^2e^{2s\varphi} )_x}{\xi
    e^{s\varphi} }v \right)^2dxdt
    \\
   & = C\int_0^T \int_\omega v^2 dxdt- 2 \int_0^T \int_\omega \xi ^2 s\varphi _t e^{2s\varphi} v^2dxdt+ 2\ve\int_0^T \int_\omega \xi ^2 e^{2s\varphi} a (v_x)^2 dxdt\\
   &
    +D_\ve \int_0^T \int_\omega  \frac{[( \xi ^2e^{2s\varphi} )_x]^2}{\xi^2
    e^{2s\varphi} }av^2 dxdt
    \end{aligned}
    \]
for some $D_\ve>0$. Hence,
    \[\begin{aligned}
&2(1-\ve-\lambda \ve)\int_0^T \int_\omega\xi ^2 e^{2s\varphi} a (v_x)^2dxdt \le C\int_0^T\int_\omega v^2 dxdt\\
 &- 2 \int_0^T \int_\omega \xi ^2 s\varphi _t
e^{2s\varphi} v^2dxdt
 +D_\ve \int_0^T \int_\omega  \frac{[( \xi ^2e^{2s\varphi} )_x]^2}{\xi^2
    e^{2s\varphi} }av^2 dxdt.
    \end{aligned}  \]
Since $x_0 \not \in \bar \omega ' $, then
    \[ \begin{aligned}
&2(1-\ve-\lambda \ve)\inf_{\omega '}a(x)\int_0^{T}\int _{\omega '}
e^{2s\varphi} (v_x)^2dxdt\\
& \le 2(1-\ve-\lambda \ve)\int_0^T
\int_{\bar\omega'} \xi ^2 e^{2s\varphi} a (v_x)^2dxdt\\&\le
2(1-\ve-\lambda \ve)\int_0^T \int_\omega \xi ^2 e^{2s\varphi} a
(v_x)^2dxdt\\&\le C\int_0^T\int_\omega v^2 dxdt- 2 \int_0^T \int_\omega \xi ^2 s\varphi_t
e^{2s\varphi} v^2dxdt
   + D_\ve\int_0^T \int_\omega  \frac{[( \xi ^2e^{2s\varphi} )_x]^2}{\xi^2
    e^{2s\varphi} }av^2 dxdt.
    \end{aligned}
\]

Finally, we show that there exists  a positive constant $C$ (still depending on $\ve$) such
that
\[
\begin{aligned}
- 2& \int_0^T \int_\omega \xi ^2 s\varphi_t e^{2s\varphi} v^2dxdt
   +D_\ve\int_0^T \int_\omega  \frac{[( \xi ^2e^{2s\varphi} )_x]^2}{\xi^2
    e^{2s\varphi} }av^2 dxdt\\
 &  \le C\int  _0 ^T \int _{\omega} v^2 dxdt,
   \end{aligned}
\]
so that the claim will follow. Indeed,
\[
|s\vp_te^{2s\varphi}|\leq c\frac{1}{s_0^{1/4}(-\max
\Upsilon)^{1/4}},
\]
$|\dot \Theta|\leq c \Theta^{5/4}$ and
\[
|s\vp_te^{2s\varphi}|\leq cs(-\Upsilon)\Theta^{5/4}e^{2s\vp}\leq
\frac{c}{\big(s(-\Upsilon)\big)^{5/4}}
\]
for some constants $c>0$ which may vary at every step.

On the other hand, $\displaystyle\frac{[( \xi ^2e^{2s\varphi}
)_x]^2}{\xi^2 e^{2s\varphi} }$ can be estimated by
\[
C\big(e^{2s\vp}+s^2(\vp_x)^2e^{2s\vp}\big)\chi_\omega,
\]
and proceeding as before, we get the claim, choosing $\ve$ small enough, namely $\ve<(1+\lambda)^{-1}$.
\end{proof}

We shall also use the following
\begin{Lemma}\label{lemma3}
Assume Hypotheses $\ref{ipotesiomega}$ and $\ref{Ass04}$. Then there exist two positive
constants $C$ and $s_0$ such that every solution $v \in \cal W$ of
\eqref{h=01} satisfies, for all $s \ge s_0$,
\[
\int_{Q_T}\left( s \Theta a (v_x)^{2} + s^3 \Theta ^3
\frac{(x-x_0)^2}{a} v^{2}\right) e^{{2s\varphi}}  dxdt\le C
\int_0^T\int_\omega v^{2} dxdt.
\]
Here $\Theta$ and $\varphi$ are as in \eqref{c_1} with $c_1$ sufficiently large.
    \end{Lemma}
Using the following non degenerate classical Carleman estimate, one has that the proof of the previous lemma is a simple adaptation of the proof of \cite[Lemma
5.1 and 5.2]{fm1}, to which we refer, also to explain why $c_1$ must be large.
\begin{Proposition}[\bf Nondegenerate nonsingular Carleman estimate]\label{classical
Carleman}
Let $z$ be the solution of
\begin{equation}
    \label{eq-z*}
    \begin{cases}
      z_t + (a  z_x) _x + \lambda \displaystyle \frac{z}{b}= h\in L^2\big((0, T)\times (A,B)\big),
      %( a \eta _x v )_x + \eta _x a v_x =:h,&
       \\
    z(t,A)= z(t,B)=0, \; t \in (0,T),
%\begin{cases}
%z(t,0) =0, & \text{ for } (WDP), \\
%\text{ or }\\ (az_x)(t,0)=0, &\text{ for } (SDP),\\
%\end{cases}& t \in (0,T).\\
    \end{cases}
    \end{equation}
where $b \in C\big([A,B]\big)$ is such that $b \ge b_0 >0$ in $[A, B]$
 and $a$ satisfies
\begin{itemize}
\item[$(a_1)$] $a\in W^{1,1}(A,B)$, $a\geq
a_0>0$ in $(A,B)$ and there exist two functions $\fg \in L^1(A,B)$,
$\fh \in W^{1,\infty}(A,B)$ and two strictly positive constants
$\fg_0$, $\fh_0$ such that $\fg(x) \ge \fg_0$ for a.e. $x$ in $[A,B]$ and
%\[
%a'(x)\left(\int_x^B\frac{\fg(t)}{\sqrt{a(t)}} dt + 2\fh_0\right)=2\sqrt{a(x)}\fg(x)  -\fh\quad \text{for a.e.} \; x \in [0,1],\]
%$\fg(x) \ge \fg_0$ for a.e. $x$ in $[0,1]$ and $\fh_0 > \displaystyle \int_A^B \frac{\fg(x)}{2\sqrt{a(x)}} dx$;
\[-\frac{a'(x)}{2\sqrt{a(x)}}\left(\int_x^B\fg(t) dt + \fh_0 \right)+ \sqrt{a(x)}\fg(x) =\fh(x)\quad \text{for a.e.} \; x \in [A,B];\]
or
\item[$(a_2)$] $a\in W^{1,\infty}(A,B)$ and $a\geq
a_0>0$ in $(A,B)$.
\end{itemize}

Then, for all $\lambda\in \R$, there exist three positive constants $C$, $r$ and $s_0$
such that for any $s>s_0$
\begin{equation}\label{carcorretta}
\begin{aligned}
\int_0^T\int_A^B \left(s\Theta (z_x)^2 + s^3 \Theta^3
 z^2\right)e^{2s\Phi}dxdt\le C\left(\int_0^T\int_A^B h^{2}e^{2s\Phi}dxdt -
(\mbox{B.T.})
\right),
\end{aligned}
\end{equation}
where
\[
(\mbox{B.T.}) = \begin{cases} \ds sr \int_0^{T}
\left[a^{3/2}e^{2s\Phi}\Theta \left(\int_x^B \fg(\tau) d\tau + \fh_0
\right) (z_x)^2\right]^{x=B}_{x=A}dt,  & \text{if (a$_1$) holds,}\\
\ds sr\int_0^T\left[ae^{2s\Phi}\Theta e^{r\zeta}(v_x)^2
\right]_{x=A}^{x=B}dt,  & \text{if (a$_2$) holds.}
\end{cases}
\]
Here the function $\Phi$ is defined as $\Phi(t,x):
=\Theta(t)\rho(x)$, where
$
\displaystyle \Theta$ is as in \eqref{c_1},
\begin{equation}\label{c_1nd}
%\Theta(t) := \frac{1}{[t(T-t)]^4} \quad \text{and} \quad
%\]
%and
%\[
\rho(x):=
\displaystyle\begin{cases} \displaystyle - r\left[\int_A^x
\frac{1}{\sqrt{a(t)}} \int_t^B
\fg(s) dsdt + \int_A^x \frac{\fh_0}{\sqrt{a(t)}}dt\right] -\mathfrak{c}, &\text{if (a$_1$) holds,}\\
\displaystyle  e^{r\zeta(x)}-\mathfrak{c}, &\text{if (a$_2$) holds,}\end{cases}
\end{equation}
and
 \[
\zeta(x)=\mathfrak{d}\int_x^B\frac{1}{a(t)}dt.
\]
Here $\fd=\|a'\|_{L^\infty(A,B)}$ and $\mathfrak{c}>0$ is
chosen in the second case in such a way that $\displaystyle
\max_{[A,B]} \rho<0$.
\end{Proposition}

\begin{proof}
     Rewrite the equation of
    \eqref{eq-z*} as $ z_t + (az_x)_x = \bar{h}, $ where $\bar{h}
    := h - \lambda \displaystyle \frac{z}{b} $. Then, applying \cite[Theorem 3.1]{fm1}, there exist
two positive constants $C$ and $s_0 >0$, such that
    \begin{equation}\label{fati1_c}
    \begin{aligned}
 \int_0^T\int_A^B \left(s\Theta (z_x)^2 + s^3 \Theta^3
 z^2\right)e^{2s\Phi}dxdt\le C\left(\int_0^T\int_A^B {\bar h}^{2}e^{2s\Phi}dxdt -
\text{(B.T.)}\right),
    \end{aligned}
    \end{equation}
    for all $s \ge s_0$.
    Using the definition of $\bar{h}$,
    the term $\ds\int_0^T\int_A^B e^{2s\Phi}\bar{h}^2
    dxdt$ can be estimated in the following way:
    \begin{equation}\label{4_c}
    \begin{aligned}
    \int_0^T\int_A^B  \bar{h}^2e^{2s\Phi}dxdt \le
    2\int_0^T\int_A^B h^2 e^{2s\Phi}dxdt
   +2\lambda^2\int_0^T\int_A^B
    \frac{z^2}{b^2} e^{2s\Phi}dxdt.
    \end{aligned}
    \end{equation}
Applying the classical Poincar\'{e} inequality to $w(t,x) :=
e^{s\Phi} z(t,x)$ and observing that
$0<\inf \Theta\leq \Theta\leq c \Theta^2$, one has
    \[
    \begin{aligned}
 2\lambda^2 \int_0^T  \int_A^B \frac{z^2}{b^2}  e^{2s\Phi}dxdt &= 2\lambda^2\int_0^T \int_A^B
\frac{w^2}{b^2} dxdt\le 2\frac{\lambda^2}{b_0^2}C \int_0^T\int_A^B  (w_x)^2 dxdt
\\&\leq C \int_0^T\int_A^B  (s^2\Theta^2  z^2+(z_x)^2 )e^{2s\Phi} dxdt \\
&\le \int_0^T\int_A^B \frac{s}{2} \Theta (z_x)^2e^{2s\Phi} dxdt+\int_0^T\int_A^B \frac{s^3}{2}\Theta^3 z^2e^{2s\Phi} dxdt,
 \end{aligned}\]
for $s$ large enough.
    Using this last inequality in (\ref{4_c}), we have
    \begin{equation}\label{fati2_c}
    \begin{aligned}
   \int_0^T\int_A^B  \bar{h}^2e^{2s\Phi}dxdt&\le
2\int_0^T\int_A^B e^{2s\Phi} h^2dxdt
   +
  \int_0^T\int_A^B \frac{s}{2} \Theta (z_x)^2e^{2s\Phi}
dxdt\\&+\int_0^T\int_A^B \frac{s^3}{2}\Theta^3 z^2e^{2s\Phi} dxdt.
    \end{aligned}
    \end{equation}
    Using this inequality
    in (\ref{fati1_c}), \eqref{carcorretta} follows immediately.
\end{proof}

In order to prove Proposition \ref{obser.}, the last result that we need is the following:
\begin{Lemma}\label{obser.regular}
Assume Hypotheses $\ref{ipotesiomega}$ and $\ref{Ass04}$. Then there exists a
positive constant $C_T$ such that every solution $v \in \cal W$ of
\eqref{h=01} satisfies
\[
\int_0^1v^2(0,x) dx \le C_T\int_0^T \int_{\omega}v^2(t,x)dxdt.
\]
  \end{Lemma}

  \begin{proof}
  Multiplying the equation of \eqref{h=01} by $v_t$ and integrating
  by parts over $(0,1)$, one has
\[
\begin{aligned}
&0 = \int_0^1\left(v_t+ (av_x)_x +\lambda \displaystyle
\frac{v}{b}\right)v_t dx= \int_0^1 \left(v_t^2+ (av_x)_xv_t +\lambda
\displaystyle \frac{vv_t}{b}\right)dx\\
& = \int_0^1v_t^2dx +
\left[av_xv_t \right]_{x=0}^{x=1} - \int_0^1av_xv_{tx} dx
+\frac{\lambda}{2}\frac{d}{dt}\int_0^1 \displaystyle
\frac{v^2}{b}dx\\
&= \int_0^1v_t^2dx -
\frac{1}{2}\frac{d}{dt}\int_0^1a(v_x)^2
+\frac{\lambda}{2}\frac{d}{dt}\int_0^1 \displaystyle \frac{v^2}{b}dx
 \\&\ge - \frac{1}{2}
\frac{d}{dt}\int_0^1 a(v_x)^2dx
+\frac{\lambda}{2}\frac{d}{dt}\int_0^1 \displaystyle
\frac{v^2}{b}dx.
\end{aligned}
\]
 Thus, the function $$t \mapsto
\displaystyle \int_0^1 a(v_x)^2 dx - \lambda\int_0^1
\frac{v^2}{b}dx$$ is non decreasing for all $t \in [0,T]$. In particular,
\[
\begin{aligned}
\int_0^1 a(v_x)^2(0,x)dx-\lambda\int_0^1
\frac{v^2(0,x)}{b(x)}dx & \le \int_0^1a(v_x)^2(t,x)dx -\lambda\int_0^1
\frac{v^2(t,x)}{b(x)}dx\\
&\text{(by Proposition \ref{PropH})}\\
& \le (1+ |\lambda| C^*)\int_0^1a(v_x)^2(t,x)dx.
\end{aligned}
\]
Integrating the previous inequality
over $\displaystyle\left[\frac{T}{4}, \frac{3T}{4} \right]$, $\Theta$ being
bounded therein, we find
%\[
%\int_0^1 a(x)(v_x)^2(0,x)dx \le \int_0^1 e^{Ct}a(x)(v_x)^2(t,x)dx \le
%e^{CT}\int_0^1a(x)(v_x)^2(t,x)dx .
%\]
%Integrating over $\left[\frac{T}{4}, \frac{3T}{4} \right]:$
\begin{equation}\label{stima2}
\begin{aligned}
&\int_0^1a(x)(v_x)^2(0,x) dx - \lambda\int_0^1
\frac{v^2(0,x)}{b(x)}dx\\
&\le \frac{2}{T}(1+ |\lambda|
C^*)\int_{\frac{T}{4}}^{\frac{3T}{4}}\int_0^1a(v_x)^2dxdt
\\&\le C_T
\int_{\frac{T}{4}}^{\frac{3T}{4}}\int_0^1s\Theta
a(v_x)^2e^{2s\varphi}dxdt.
\end{aligned}
\end{equation}
Hence, from the previous inequality and Lemma \ref{lemma3}, if $\lambda <0$
\[
\int_0^1a (v_x)^2(0,x) dx \le \int_0^1a(v_x)^2(0,x) dx -
\lambda\int_0^1 \frac{v^2(0,x)}{b(x)}dx\le C \int_0^T
\int_{\omega}v^2dxdt
\]
for some positive constant $C>0$.

If $\lambda >0$, using again Lemma \ref{lemma3} and
\eqref{stima2}, one has
\begin{equation}\label{chissà}
\int_0^1a(v_x)^2(0,x) dx - \lambda\int_0^1
\frac{v^2(0,x)}{b(x)}dx\le C \int_0^T \int_{\omega}v^2dxdt.
\end{equation}
Hence, by \eqref{1} and \eqref{chissà}, we have
\[
\begin{aligned}
\int_0^1a(v_x)^2(0,x) dx &\le \lambda\int_0^1
\frac{v^2(0,x)}{b(x)}dx + C \int_0^T \int_{\omega}v^2dxdt \\ &\le
\lambda C^* \int_0^1a(v_x)^2(0,x) dx + C \int_0^T
\int_{\omega}v^2dxdt.
\end{aligned}\]
Thus
\[
(1-\lambda C^* ) \int_0^1a(v_x)^2(0,x) dx \le C\int_0^T
\int_{\omega}v^2dxdt,
\]
 for a positive constant $C$. In every case, there exists $C >0$
 such that
 \begin{equation}\label{stum}
\int_0^1a(v_x)^2(0,x) dx \le C\int_0^T \int_{\omega}v^2dxdt.
 \end{equation}
 The Hardy- Poincar\'{e} inequality (see Proposition \ref{HP})
and \eqref{stum} imply that
\[
\begin{aligned}
\int_0^1 \left(\frac{a}{(x-x_0)^2}\right)^{1/3}v^2(0,x)dx &\le
\int_0^1 \frac{p}{(x-x_0)^2} v^2(0,x)dx  \\&\le C_{HP} \int_0^1
p(v_x)^2(0,x) dx \\&\le cC_{HP}  \int_0^1a(v_x)^2(0,x) dx \\
&\le C
\int_0^T\int_{\omega}v^2dxdt,
\end{aligned}
\]
for a positive constant $C$.  Here $p(x) = (a(x)|x-x_0|^4)^{1/3}$ if $K_1 > \displaystyle\frac{4}{3}$, while $p(x) =\ds |x-x_0|^{4/3}\max_{[0,1]}a^{1/3}$ otherwise, and $c, C$ are obtained by Lemma \ref{Lemma 2.1}.

Again by Lemma \ref{Lemma 2.1}, we have
\[\left(\frac{a(x)}{(x-x_0)^2}\right)^{1/3}\ge C_3:=\min\left\{\left(\frac{a(1)}{(1-x_0)^2}\right)^{1/3},
\left(\frac{a(0)}{x_0^2}\right)^{1/3}\right\} >0.
\]
Hence
\[
C_3\int_0^1v(0,x)^2dx \le C \int_0^T\int_{\omega}v^2dxdt\] and the
claim follows.
\end{proof} 

\begin{proof}[Proof of Proposition $\ref{obser.}$] It follows by a density argument as for the proof of
\cite[Proposition 4.1]{fm}.
\end{proof}

\end{document}